\documentclass[11pt,letterpaper]{amsart}
\usepackage{graphicx,mathtools,enumerate}
\usepackage{amssymb}
\usepackage{epstopdf}
\usepackage{varwidth}
\usepackage{cancel}
\usepackage[colorlinks=true, pdfstartview=FitV, linkcolor=blue, citecolor=blue, urlcolor=blue]{hyperref}
\usepackage{tikz-cd}
\usepackage{caption}
\usepackage{subcaption}
\usepackage{mathrsfs}
\usepackage{comment}
\usepackage{framed}
\usepackage{color}
\input{xy}
\xyoption{poly}
\xyoption{all}

\definecolor{shadecolor}{gray}{0.875}

\usepackage[margin=2.5cm]{geometry}

\usepackage{color}

\newcommand{\Z}{{\mathbb Z}}

\def\Ass{\operatorname{Ass}}

\def\sat{{\rm sat}}
\def\sdef{\operatorname{sdef}}

\def\sl{\operatorname{s\ell}}

\def\NP{\operatorname{NP}}
\def\SP{\operatorname {SP}}
\def\isdef{\operatorname {isdef}}

\def\inv{^{-1}}

\let\frak\mathfrak

\newcommand{\factor}[2]{\left. \raise 1pt\hbox{\ensuremath{#1}} \right/
        \hskip -2pt\raise -3pt\hbox{\ensuremath{#2}}}

\theoremstyle{plain} 
\newtheorem{thm}{Theorem}[section]

\newtheorem*{introthm*}{Theorem}
\newtheorem{question}{Question}
\newtheorem{cor}[thm]{Corollary}
\newtheorem{lem}[thm]{Lemma}
\newtheorem{prop}[thm]{Proposition}

\theoremstyle{definition}
\newtheorem{defn}[thm]{Definition}

\newtheorem{ex}[thm]{Example}

\newtheorem{rem}[thm]{Remark}

\numberwithin{equation}{section}  

\makeatletter  
\@namedef{subjclassname@2020}{%
  \textup{2020} Mathematics Subject Classification}
\makeatother

\title{Symbolic Defect of Monomial Ideals}
\author[B. Oltsik]{Benjamin Oltsik}
\address{Department of Mathematics, University of Connecticut, Storrs, CT 06269}
\email{benjamin.oltsik@uconn.edu}

\begin{document}
\begin{abstract} 
Given a monomial ideal $I$, we study two functions that quantify ways to measure the difference between symbolic powers and usual powers of $I$. In many cases we determine the asymptotic growth rate of these two functions. We also perform explicit computations by using the symbolic polyhedron. 
\end{abstract}

\maketitle

\section{Introduction} Symbolic powers of ideals $I$ (Definition \ref{def:sympower}), denoted $I^{(n)}$, in a Noetherian commutative ring $R$ have long been a topic of interest in commutative algebra \cite{SymPowerNotes}.  They are directly related to the study of primary decomposition of ideals. In algebraic geometry, a special case of symbolic powers appear as ideals of functions which vanish on a given variety to a certain order.  
The definition that we work with is (see Remark \ref{rem:otherdefinition}):
\[I^{(n)} \coloneqq \bigcap_{\mathfrak p \in \Ass(R/I)} I^n R_{\mathfrak p} \cap R.\]
It is natural to study the minimal set of generators for symbolic powers, but calculating such minimal generating sets is often difficult.  Even counting the minimal number of generators $\mu(I^{(n)})$ has proven challenging.  Dao-Montaño \cite{DM}, study this question asymptotically via the \textit{symbolic analytic spread}, which is essentially the growth rate of the function $n\mapsto \mu(I^{(n)})$.  In this paper, our first object of focus is the \textit{symbolic defect} function, originally defined in \cite{GGSV}:

\[\sdef_I(n)\coloneqq\mu\left(\factor{I^{(n)}}{I^n}\right).\]  We aim to study the asymptotic growth rate of symbolic defect for monomial ideals.
We will use the following asymptotic notation: for functions $f$ and $g$, we say
\begin{itemize}
    \item $f(n) = O(g(n))$ if $\displaystyle\lim_{n \to \infty} \frac {f(n)}{g(n)} < \infty$,
    \item $f(n) \sim g(n)$ if $\displaystyle\lim_{n \to \infty} \frac {f(n)}{g(n)} = 1$,
    \item $f(n) = \Theta(g(n))$ if $0 < \displaystyle\lim_{n \to \infty} \frac {f(n)}{g(n)} < \infty$.
\end{itemize}

A result of Drabkin and Guerrieri \cite{DG} implies that the symbolic defect function is eventually quasi-polynomial for monomial ideals.  In Section 3, we show,

\begin{thm}\label{thm:bigOsymdef} Let $I$ be a monomial ideal in $k[x_1, \ldots, x_r]$ without embedded primes.  Then $\sdef_I(n) = O(n^{r-2})$.
\end{thm}

One of the classical avenues for studying asymptotic questions is via convex polyhedra such as the the Newton and symbolic polyhedra \cite{HS}, \cite{symPowerMon}. This approach inspires the study of a new invariant, the \textit{integral symbolic defect} (Definition \ref{isdef}):
\[{\rm isdef}_I(n)\coloneqq\mu\left(\factor{\overline{I^{(n)}}}{\overline{I^n}}\right).\]  
While there is no immediate comparison between $\sdef_I(n)$ and $\isdef_I(n)$, we prove that their $O$-asymptotic growth is the same in a special case:
\begin{thm}\label{thm:isdef} Let $I$ be a monomial ideal in $R = k[x_1, \ldots, x_r]$ such that $I^{(n)} = (I^n)^{\operatorname{sat}}$ for all $n\geq 1$ (e.g., when $\dim (R/I) = 1$).  Then $\isdef_I(n) = O(n^{r-2})$.
\end{thm}

In Sections 5 and 6 we turn our attention to the problem of computing $\sdef_I(n)$, or at least computing its leading term and a quasi-period. We say \textit{a} quasi-period because it may not be minimal, but it is a multiple of the minimal quasi-period.  In Section 6 we do this for a special class of monomial ideals.

\begin{ex}\label{ex:sdefCalc} Let $I = (x^a, y) \cap (y^b, z) \cap (z^c, x)$.  Then 
\begin{enumerate}[(i)]
    \item $\sdef_I(n) \sim (\alpha + \beta + \gamma)n$, where $( \alpha, \beta, \gamma) = \left( \frac{a(bc - b + 1)}{abc + 1}, \frac{b(ac - c + 1)}{abc + 1}, \frac{c(ab - a + 1)}{abc + 1} \right)$.
\item A quasi-period of $\sdef_I(n)$ is $abc+1$.
\end{enumerate}
\end{ex}

In Section 5, we use the symbolic polyhedron to calculate the exact symbolic defect function of $I = (xy, yz, xz)$. This strengthens the case $a=b=c=1$ of Example \ref{ex:sdefCalc}, as well as recovers the results of \cite{Montero}:
\begin{ex}\label{prop:easyExsdef}  For $I = (xy, xz, yz)$,
\[
    \sdef_I(n) = \begin{cases} \frac32 n - 2 & n \equiv 0 \pmod 2\\
    \frac 32n - \frac 32 & n \equiv 1 \pmod 2
    \end{cases}.
\]
\end{ex}

\noindent We end with a list of open questions in Section 7.

\section{Background}\label{sec::background}
\begin{defn}\label{def:sympower} Let $R$ be a commutative, unital, Noetherian ring, and let $I$ be an ideal.  We define the $n$-th \textit{symbolic power} of $I$ to be
\[
    I^{(n)} := \bigcap_{\mathfrak p \in \Ass(R/I)} I^n R_{\mathfrak p} \cap R.
\]
    
\end{defn}

\begin{rem}\label{rem:otherdefinition}
    The literature has a competing definition where the intersection is taken over just the minimal primes of $I$.  In the case where $I$ has no embedded primes, these definitions coincide, but they may differ in general.  For example, $I = (x^2, xy)$ has $I^{(n)} = (x^2, xy)^n$ by our definition, but with the alternate definition, the symbolic powers would be $(x)^n$.
\end{rem} 

Another way to represent symbolic powers is via a saturation.  Indeed, a result from McAdam \cite{McAdam} states the associated primes of $R/I^n$ eventually stabilize for large enough $n$.  Let $A^*(R/I)$ be this stabilized set, called the \textit{asymptotic primes} of $I$.  
\begin{lem}\cite[Lemma 2.2]{HJKN} Let \[J = \bigcap\limits_{\substack{\mathfrak p \in A^*(R/I) \setminus \Ass(R/I) \\ \operatorname{depth}(\mathfrak p, R/I) \ge 1}} \mathfrak p.\]  
Then,
\[
    I^{(n)} = (I^n : J^\infty) = \bigcup_{i \ge 1} (I^n : J^i).
\]
    
\end{lem}

\begin{rem}\label{rem:satSym}
Note that if $J = \mathfrak m$, then $I^{(n)} = (I^n)^{\text{sat}}$ for all $n$.
\end{rem}

To ordinary powers and symbolic powers, we define the two graded algebras: $\mathcal R(I) = \oplus_{n \ge 0} I^nt^n$, the Rees algebra, and $\mathcal R_s(I) = \oplus_{n \ge 0} I^{(n)}t^n$, the symbolic Rees algebra.  While the former is always Noetherian, in general, the latter may not be (see Example 2.4 in \cite{GS}). Denote $\ell(I) = \dim (\mathcal R(I) \otimes R/\mathfrak m)$, and $s\ell(I) = \dim(\mathcal R_s(I) \otimes R/\mathfrak m)$, called the {\it analytic spread of $I$} and {\it symbolic analytic spread of $I$}, respectively.

For any $R$-module $M$, denote by $\mu(M)$ the minimal number of generators of $M$.  For ideals $I$, the growth rate of $\mu(I^n)$ and $\mu(I^{(n)})$ have been studied.  Indeed, from Nakayama's lemma, we have that $\ell(I) = \inf\{t \in \mathbb R : \mu(I^n) = O(n^{t-1})\}$, and, for when $\mathcal R_s(I)$ is Noetherian, $s\ell(I) = \inf\{t \in \mathbb R : \mu(I^n) = O(n^{t-1})\}$ \cite{DM} are well-known invariants, particularly $\ell(I)$.  We wish to expand the study of these invariants in regards to the following function:

\begin{defn}
[Galetto-Geramita-Shin-Van Tuyl \cite{GGSV}]
\label{def:sdef}
 Let $R$ be either local or graded with unique homogeneous maximal ideal, $\mathfrak m$, and residue field $k$.  The {\em symbolic defect function} of a homogeneous ideal $I$ is the numerical function
\[
\sdef_I:\Z_{\ge 0} \to \Z_{\ge 0}, \quad \sdef_I(n):=\mu\left(\factor{{I^{(n)}}}{{I^n}}\right)=\dim_k\left(\frac{I^{(n)}}{I^n +\mathfrak{m}I^{(n)}} \right) < \infty.
\]
\end{defn}
This function is designed to be a measurement of ``closeness" between $I^{(n)}$ and $I^n$. 

A result of Drabkin and Guerrieri \cite[Theorem 2.4]{DG} states that, for a homogeneous ideal $I$ such that $\mathcal R_s(I)$ is Noetherian, $\sdef_I(n)$ is eventually quasi-polynomial. Note that this result is not dependent on field characteristic or cardinality.  A useful fact about monomial ideals is, when we view them as homogeneous ideals in $k[x_1, \ldots, x_r]$, that $\mathcal R_s(I)$ is always Noetherian \cite[Theorem 3.2]{HHT}.  This allows us to invoke Drabkin and Guerrieri's theorem to conclude that $\sdef_I(n)$ is eventually quasi-polynomial for any monomial ideal.  In this paper, we wish to determine an upper-bound of the degree(s) of the quasi-polynomial in this paper, specifically for monomial ideals in polynomial rings.  

For monomial ideals, there are many other results at our disposal regarding symbolic powers.  For example, the following allows for much easier calculations:

\begin{lem}\label{primdecomppowers} \cite[Lemma 3.1]{HHT} Let $I$ be a monomial ideal with monomial primary decomposition, $I = Q_1 \cap \cdots \cap Q_s$.  Set $\max(I)$ to be the set of maximal associated primes, and, for each $\mathfrak p \in \max(I)$, let $Q_{\subseteq \mathfrak p} = \bigcap\limits_{\sqrt{Q_i} \subseteq \mathfrak p} Q_i$.  Then,
\[
    I^{(n)} = \bigcap_{\mathfrak p \in \max(I)} (Q_{\subseteq \mathfrak p})^n.
\]

In particular, if $I$ does not have embedded primes, then $I^{(n)} = Q_1^n \cap \cdots \cap Q_s^n$.
\end{lem}

\begin{ex}\label{ex:maxAssoc}
If $I$ is a monomial ideal in $k[x_1,\ldots,x_r]$ and if $\frak m=(x_1,\ldots,x_r)$ is associated to $I$, then $I^{(n)}=I^n$ for all $n\geq 1$. In particular, if $r=2$ then $I^{(n)}=I^n$ for all $n\geq 1$.
\end{ex}

We now introduce a few tools to help work with monomial ideals.
\begin{defn} \label{def:NP} Let $R = k[x_1, \ldots, x_r] = k[\mathbf{x}]$, and let $I$ be a monomial ideal.  We define the \textit{Newton polyhedron of $I$}, denoted $\NP(I)$, as the convex hull of exponent vectors for monomials in $I$:
\[
    \NP(I) = \operatorname{cvxhull}\{\mathbf b \in \Z_{\ge 0}^n : {\mathbf x}^\mathbf b \in I\},
\]
\end{defn}
where $\mathbf b = (b_1, \ldots, b_r)$ and $\mathbf x^{\mathbf b} = x_1^{b_1}\cdots x_r^{b_r}$.  In this paper, we will make no distinction between a monomial and its corresponding exponent vector in $\mathbb R^r$.

The Newton polyhedron allows the study of monomial ideals from the perspective of Euclidean geometry.  An important property to note is that $\NP(I^n) = n \NP(I)$ for all $n \ge 1$.  This allows $\NP(I)$ to be quite useful when studying powers of monomial ideals.  We can now define another polyhedron to study symbolic powers:

\begin{defn}\label{def:SP}\cite[Definition 3.2]{CDFFHSTY} Let $I$ and $Q_i$ be as in Lemma \ref{primdecomppowers}.  Then the \textit{symbolic polyhedron} of $I$, denoted $\SP(I)$, is
\[
    \SP(I) = \bigcap_{\mathfrak p \in \max(I)} \NP(Q_{\subseteq \mathfrak p}).
\]

In particular, if $I$ has no embedded primes, then $\SP(I) = \NP(Q_1) \cap \cdots \cap \NP(Q_s)$.

\end{defn}
\begin{rem} For Newton and symbolic polyhedra of monomial ideals in $k[x_1, \ldots, x_r]$, we use variables $u_1, u_2, \ldots, u_r$ to describe coordinates in $\mathbb R^r$.  If $r = 3$, we use coordinates $u, v, w$.

\end{rem}

\begin{ex} \label{ex:easyExample} Let $I = (xy, xz, yz) \subseteq k[x, y, z]$.  This is one of the simplest examples of ideals with non-trivial symbolic powers; for example, $I^2 = (x^2y^2, x^2yz, xy^2z, xyz^2, x^2z^2, y^2z^2)$, but $I^{(2)} = (x^2y^2, x^2z^2, y^2z^2, xyz)$. 
 Thus, $\NP(I) \ne \SP(I)$.  In particular,
\[
    \NP(I) = \begin{cases} u + v \ge 1\\
    u + w \ge 1\\
    v + w \ge 1\\
    u + v + w \ge 2\\
        u, v, w \ge 0,
    \end{cases}
\quad\text{ and }\quad 
    \SP(I) = \begin{cases} u + v \ge 1\\
    u + w \ge 1\\
    v + w \ge 1\\
    u, v, w \ge 0.
    \end{cases}
\]
Figures \ref{fig:easyExNP} and \ref{fig:easyExSP} are graphs of these two bodies.  In particular, note that $(1/2, 1/2, 1/2)$ is in $\SP(I)$ but not in $\NP(I)$, which corresponds to $xyz \in I^{(2)}\setminus I^2$.  In Section 5, we study this ideal further.
\end{ex}

\begin{figure}
\begin{minipage}{.5\textwidth}
  \centering
  \includegraphics[width=.6\linewidth]{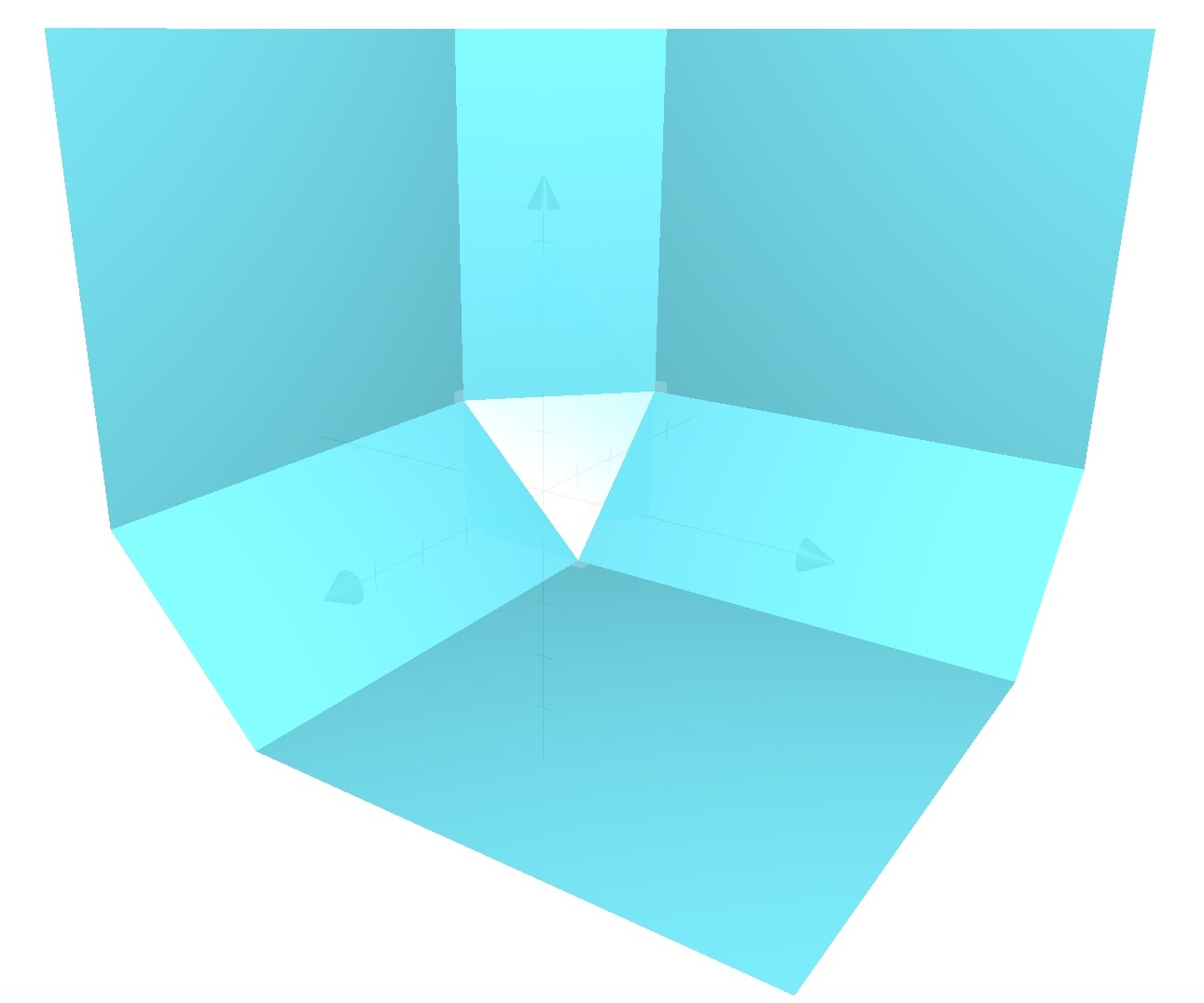}
  \captionof{figure}{${\rm NP}(xy, yz, xz)$}
  \label{fig:easyExNP}
\end{minipage}%
\begin{minipage}{.5\textwidth}
  \centering
  \includegraphics[width=.6\linewidth]{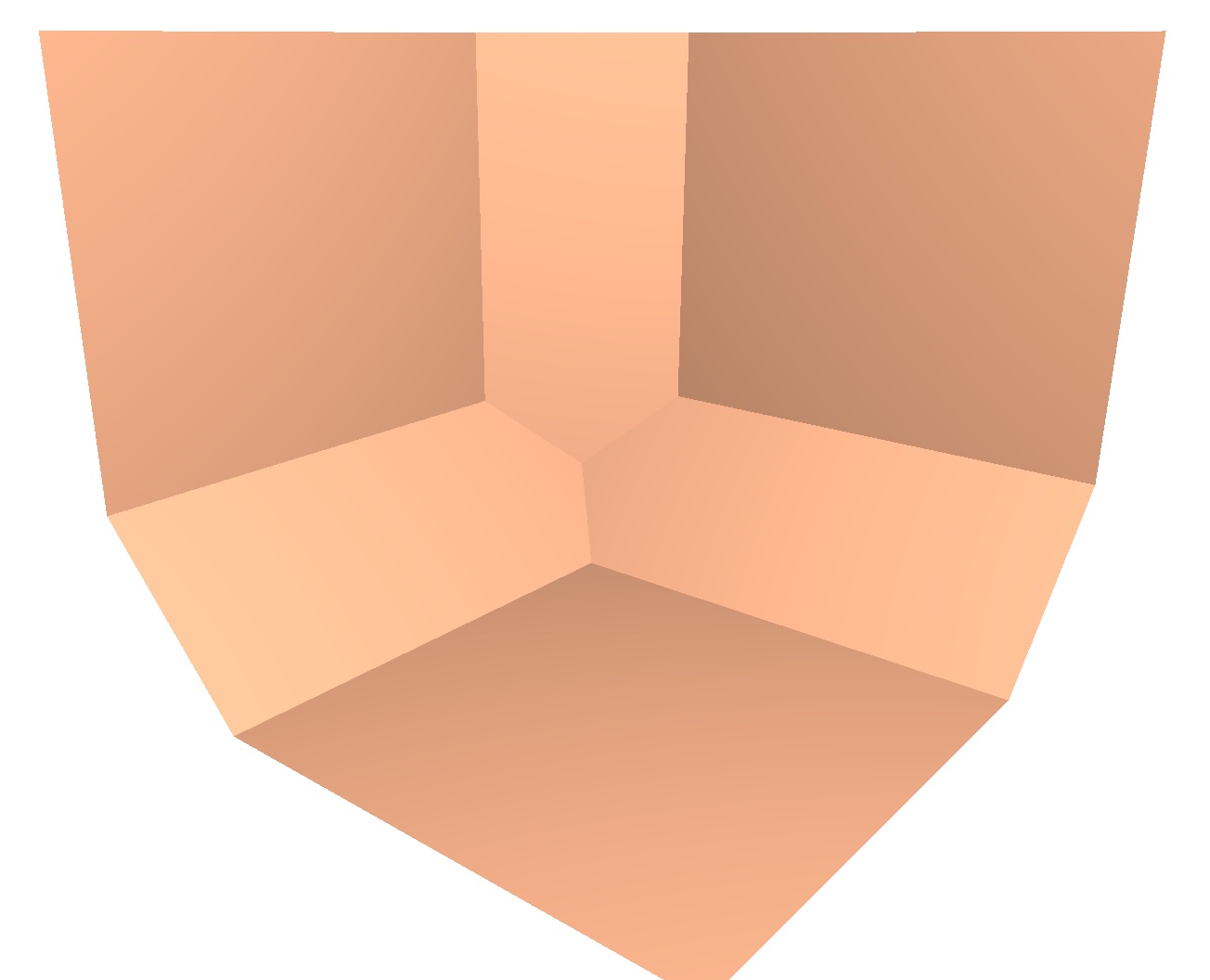}
  \captionof{figure}{$\SP(xy, yz, xz)$}
  \label{fig:easyExSP}
\end{minipage}
\end{figure}

\begin{ex} \label{ex:purePowers}\cite[Lemma 3.7]{CDFFHSTY}.  Let $I$ be an ideal without embedded primes such that $I$ is an intersection of pure powers.  That is, $I = \bigcap\limits_{1 \le i \le s} Q_i$, where $Q_i = (x_1^{a_{i1}}, x_2^{a_{i2}}, \ldots, x_r^{a_{ir}})$, where $a_{ij} \in \Z_{>0} \cup \{-\infty\}$, with the convention that $x_i^{-\infty} = 0$, and $\sqrt{Q_i}$ distinct for each $i$.  Then, $\SP(I)$ consists of the non-negative solutions of:
\[
    \begin{cases} \frac1{a_{11}} u_1 + \frac 1{a_{12}}u_2 + \cdots + \frac 1{a_{1r}} u_r \ge 1,\\
    \vdots\\
    \frac1{a_{s1}}u_1 + \frac 1{a_{s2}}u_2 + \cdots + \frac 1{a_{sr}} u_r \ge 1,

    \end{cases}
\]
where $\frac1{-\infty} = 0$.
    
\end{ex}

We note that containment in either polyhedron \textit{does not} imply containment in the ideal.  That is, if ${\mathbf x}^{\mathbf b} \in I^n$, then $\mathbf b \in n\NP(I) \cap \Z^r_{\ge 0}$, but the converse may not be true.  Similarly, if ${\mathbf x}^{\mathbf b} \in I^{(n)}$, then $\mathbf b \in n\SP(I) \cap \Z^r_{\ge 0}$, without the converse necessarily holding.  However, we can extend $I$ so that the converse always holds:

\begin{defn}\label{def:intCl} Let $I$ be an ideal in a ring $R$. We say an element $r \in R$ is \textit{integral} over $I$ if $r$ satisfies a monic, polynomial equation of the form,
\[
    r^n + a_1r^{n-1} + \cdots + a_n = 0
\]
where $a_i \in I^i$.  The \textit{integral closure} of $I$, denoted $\overline{I}$, is the set of all elements in $R$ that are integral over $I$.  If $I = \overline{I}$, we say $I$ is \textit{integrally closed}. 
    
\end{defn}

The integral closure of an ideal is, itself, an ideal.  Furthermore, if $I$ is a monomial ideal, then $\overline{I}$ is also a monomial ideal \cite{HS}.  In fact, we can do even better:

\begin{prop}\label{prop:intClForm} Let $I \subseteq k[x_1, \ldots, x_r] = k[\mathbf{x}]$ be a monomial ideal.  Then,
\[
    \overline{I} = ( \mathbf{x}^{\mathbf b} : \mathbf{x}^{a \mathbf b} \in I^a \text{ for some } a \ge 1).
\]
    
\end{prop}

A big advantage of working with integral closure is that $\NP(\overline I) = \NP(I)$, and ${\mathbf x}^{\mathbf b} \in \overline{I^n}$ if and only if $\mathbf b \in n \NP(I)$.  We will see soon in Theorem \ref{thm:SPIcont} that a similar result holds for $\SP(I)$.

There is one more way to define $\SP(I)$, and that is via a limiting body:

\begin{thm}\label{limbodySPI}\cite[Corollary 3.16]{CDFFHSTY} For a monomial ideal $I$,
\[
    \SP(I) = \bigcup_{n \to \infty} \frac 1n \NP(I^{(n)}).
\]
\end{thm}

For further reading on limiting bodies and their relation with analytic spread, see \cite{HN}.

\section{Growth of Symbolic Defect of Monomial Ideals}

This brief section will determine an upper-bound for all monomial ideals without embedded primes.  In a paper by Dao and Montaño \cite{DM}, they prove the following result:

\begin{thm}\label{thm:slbound} \cite[Theorem 4.1]{DM} Let $I$ be a monomial ideal without embedded primes in $k[x_1, \ldots, x_r]$.  Then,
\[
    \sl(I) \le r - \left\lfloor \frac{r - 1}{\operatorname{bigHeight} I}\right\rfloor,
\]
where $\operatorname{bigHeight} I$ denotes the maximal height of an associated prime of $I$.

\end{thm}

Note that in \cite{DM}, they employ the alternate definition of symbolic powers, where the intersection described in Definition \ref{def:sympower} is taken over the minimal primes of $I$.  In the case without embedded primes, our definition and their definition coincide. 

We are now ready to prove Theorem \ref{thm:bigOsymdef}:

\begin{proof}[Proof of Theorem \ref{thm:bigOsymdef}] If $\mathfrak m$ is an associated prime, then by Example \ref{ex:maxAssoc}, $\sdef_I(n) = 0$.  Otherwise, the bound in Theorem \ref{thm:slbound} is no more than $r - 1$.  Since $\sdef_I(n) \le \mu(I^{(n)})$, the conclusion follows.
\end{proof}

In the next section we prove a similar result for classes of ideals that have embedded primes.

\section{Integral Closure and Symbolic Defect}

By the nature of Newton polyhedra, a monomial's exponent vector is within the Newton polyhedron of a monomial ideal if and only if the monomial is in the integral closure of the ideal.  With this in mind, it is natural to expand the study of symbolic polyhedra and consider the interactions with integral closure.

\begin{defn}\label{isdef} Let $I$ be an ideal of a commutative ring $R$.  We define the \textit{integral symbolic defect of $I$} as
\[
    \operatorname{isdef}_I(n) := \mu\left(\factor{\overline{I^{(n)}}}{\overline{I^n}}\right) = \dim_k\left(\frac{\overline{I^{(n)}}}{\overline{I^n} +\mathfrak{m}\overline{I^{(n)}}} \right)
\]
\end{defn}

The following result will be essential in calculating $\operatorname{isdef}_I(n)$:

\begin{thm}\label{thm:SPIcont} Let $I$ be a monomial ideal in $k[x_1, \ldots, x_r] = k[\mathbf x]$.  Then ${\mathbf x}^{\mathbf u} \in \overline{I^{(m)}}$ if and only if $\mathbf u \in (m \SP(I)) \cap \Z^r_{\ge 0}$.
    
\end{thm}

\begin{proof} $(\Leftarrow)$ Assume $\mathbf u \in (m \SP(I)) \cap \Z^r_{\ge 0}$.  Then by Theorem \ref{limbodySPI}, 
\[
\mathbf u \in \left( \bigcup_{n \to \infty} \frac mn \NP(I^{(n)}) \right) \cap \Z^r_{\ge 0} = \bigcup_{n \to \infty} \frac mn \NP(I^{(n)}) \cap \Z^r_{\ge 0},
\] so $\mathbf u \in \frac mn \NP(I^{(n)}) \cap \Z^r_{\ge 0}$ for some $n$.  Note that, for any $t \ge 1$, 
$$\frac mn \NP(I^{(n)}) = \frac m{nt}\NP((I^{(n)})^t) \subseteq \frac m{nt} \NP(I^{(nt)}).$$  Thus, we may assume that $m \le n$.  It follows that 
$$\frac mn \NP(I^{(n)}) \subseteq \frac mn \NP(I^{(m)}) \subseteq \NP(I^{(m)}).$$  So $\mathbf u \in \NP(I^{(m)}) \cap \Z^r_{\ge 0}$, and we are done.




$(\Rightarrow)$ If  ${\mathbf x}^{\mathbf u} \in \overline{I^{(m)}}$, then  
$$\mathbf u \in \NP(I^{(m)})\cap \Z^r_{\ge 0} = m \left(\frac 1m \NP(I^{(m)})\cap \Z^r_{\ge 0}\right) \subseteq
m \left( \bigcup\limits_{n \to \infty} \frac 1n \NP(I^{(n)}) \cap \Z^r_{\ge 0}\right).$$
\end{proof}

\begin{rem} Geometrically, integral symbolic defect corresponds to the number of points in $n\operatorname{SP}(I)\setminus n\NP(I)$ that are minimal.  That is, the number of points $\mathbf u \in \mathbb{R}^r$ such that $\mathbf u \in n\SP(I) \setminus n\NP(I)$ but $\mathbf u - \mathbf e_i \not\in n\SP(I) \setminus n\NP(I)$, for every standard unit vector, $\mathbf e_i$.  This includes any points that are not in $I^{(n)}$ but are in $n\operatorname{SP}(I)$.
\end{rem}

Finding $\isdef_I(n)$ amounts entirely to a geometric problem.  Precisely, we desire to count the generators in the affine semigroup corresponding to the symbolic polyhedron that are not contained in the Newton polyhedron.  As such, it is easier to calculate $\isdef_I(n)$ over $\sdef_I(n)$.  Directly comparing $\sdef_I(n)$ and $\isdef_I(n)$, however, still proves to be challenging.  This is a bit surprising, considering the result from Hà-Nguy$\tilde{\text{\^e}}$n \cite[Corollary 4.4]{HN}, implicating $\mu(I^n) = \Theta(\mu(\overline{I^n}))$ and $\mu(I^{(n)}) = \Theta(\mu(\overline{I^{(n)}}))$.

\begin{ex} If $I$ is a monomial ideal, often one will find that $\mu(I) \le \mu(\overline I)$, but this is not always the case.  An counterexample due to Michael DiPasquale is the ideal $I = (x^3,x^2yz^2,xy^2z^2,x^2y^2z,y^3)$, whose integral closure is $\overline{I} = (x^3,x^2y,xy^2,y^3)$.  This is also an example of a monomial ideal whose integral closure is primary, but the ideal itself is not, although the converse is known to be true.
\end{ex}

Despite Hà-Nguy$\tilde{\text{\^e}}$n's aforementioned result, it is unknown whether $\sdef_I$ and $\isdef_I$ grow similarly.  Certainly, if $I$ is normal and the symbolic powers are integrally closed, then $\sdef_I(n) = \isdef_I(n)$ for all $n$.  Furthermore, the following conditions are sufficient, albeit artificial:
\begin{prop}
\begin{enumerate} 

    \item If $\mu(\overline{I^{(n)}}) \not\sim \dim_k K_n - \mu(\overline{I^n})$, where $K_n = \mathfrak m \overline{I^{(n)}} \cap \overline{I^n} / \mathfrak m \overline{I^n}$, then $\sdef_I(n) = O(\isdef_I(n))$.
    \medskip
    
    \item If $\mu(I^{(n)} \cap \overline{I^n} / I^n) = O(\isdef_I(n))$, then $\sdef_I(n) = \Theta(\isdef_I(n))$.

\end{enumerate}
\end{prop}

\begin{proof} We will assume that for $n \gg 0$, $\sdef_I(n)$ and $\isdef_I(n)$ are non-zero, for otherwise the results are immediate.

For (1), consider the short exact sequence,
\[
0 \to \overline{I^n} \to \overline{I^{(n)}} \to \overline{I^{(n)}}/\overline{I^n}\to 0.
\]
By tensoring with $k$ and applying Nakayama's lemma, we get
\[
    \mu(\overline{I^n}) + \isdef_I(n) = \dim_k{K_n} + \mu(\overline{I^{(n)}}),
\]
and thus,
\[
    \isdef_I(n) = \dim_k{K_n} + \mu(\overline{I^{(n)}}) - \mu(\overline{I^n})
\]

Let $P(n) = \dim_k{K_n} - \mu(\overline{I^n})$.  If we consider $F_n = \frac{\sdef_I(n)}{\isdef_I(n)}$, then
\[
F_n \le \frac{\mu(I^{(n)})}{\mu(\overline{I^{(n)}}) +P(n)}.
\]
Since $\mu(I^{(n)}) = \Theta(\mu(\overline{I^{(n)}}))$ \cite[Corollary 4.4]{HN}, we can consider the numerator of the latter fraction to be a polynomial of degree $n^{s\ell(I) - 1}$, and the denominator to be $P(n)$ plus a polynomial of degree $n^{s\ell(I) - 1}$.  If $\deg(P(n)) > s\ell(I) - 1$, then $F_n \to 0$.  If $\deg(P(n)) < s\ell(I) - 1$, $F_n$ is bounded above by a finite value.  If $\deg(P(n)) = s\ell(I) - 1$ but the leading coefficients are not the same, then $F_n$ is still bounded above by a finite value.  Thus, the only case where $F_n$ is not bounded above with this inequality is if $P(n) \sim \mu(\overline{I^{(n)}})$.

For (2), consider the diagram:
\[\xymatrix@C=40pt{
    &0\ar[d]&0\ar[d]&\\
 0\ar[r]&I^n\ar[r] \ar[d] & \overline{I^{n}} \ar[r]\ar[d] & \overline{I^n}/I^n \ar[r]\ar[d] &0\\
  0\ar[r]&I^{(n)} \ar[r] & \overline{I^{(n)}} \ar[r] & \overline{I^{(n)}}/I^{(n)} \ar[r] &0\\
}\]
Apply the snake lemma to obtain
\[
    0 \to \overline{I^n} \cap I^{(n)}/I^n \to I^{(n)}/I^n \to \overline{I^{(n)}}/\overline{I^n} \to C_n \to 0,
\]
where $C_n = \operatorname{coker}(\overline{I^n}/I^n \to \overline{I^{(n)}}/I^{(n)})$.  Tensoring by $k$ and applying Nakayama's lemma again, we have
$\sdef_I(n) \le \mu(\overline{I^n} \cap I^{(n)}/I^n) + \isdef_I(n)$, so
\[
    F_n \le \frac{\mu(\overline{I^n} \cap I^{(n)}/I^n)}{\isdef_I(n)} + 1,
\]
and the claim follows.

\end{proof}
Both of the above conditions are not easy to verify, and it is  unknown what particular ideals satisfy the conditions.  As such, a precise, verifiable relation between $\sdef_I(n)$ and $\isdef_I(n)$ is unknown.

Due to the geometric nature of $\isdef_I(n)$, we can more easily study the growth of $\isdef_I(n)$.  First, we establish the following lemma:

\begin{lem} Let $I, J \subset k[x_1, \ldots, x_r]$ be monomial ideals such that $I \subsetneq J$.  We will say the number of monomials in $J/I$ is finite if there are finitely many elements of $J/I$ with a residue class represented by a monomial.  Then the number of monomials in $M = J/I$ is finite if and only if $\Ass(J/I) = \{\mathfrak m\}$.  Furthermore, if $J/I$ has finitely many monomials, so does $\overline{J}/\overline{I}$.
\end{lem}

\begin{proof} Assume $\Ass(M) \ne \{\mathfrak m\}$.  Since $J \ne I$, $\Ass(M) \ne \emptyset$, so there exists $\mathfrak p \in \Ass(M)$, $\mathfrak p \ne \mathfrak m$.  Let $m$ be a monomial such that $\operatorname{Ann}(m + I) = \mathfrak p$, and let $x_i \not\in \mathfrak p$.  Then, for all $f \ge 1$, $x_i^f m \not\in I$.  Thus, there are infinitely many monomials in $M$.

Conversely, assume $\Ass(M) = \{\mathfrak m\}$.  Then, for every $x_i$, there exists $f_i$ such that, for any $m \in J$, $x_i^{f_i}m \in I$.  That is, there is an upper-bound for each exponent of a non-zero monomial in $M$, hence only finitely many monomials exist.

Now assume $J/I$ has finitely many monomials so that $\Ass(J/I) = \{\mathfrak m\}$.  If $\overline{J} = \overline{I}$, the module is trivial.  Otherwise, letting $J = (n_1, \ldots, n_s)$, this means that, for each $i \in \{1, \ldots, s\}$ and $j \in \{1, \ldots, r\}$, there exists $b_{ij}$ such that $x_j^{b_{ij}}n_i \in I$.  Now, let $m \in \overline{J}$.  Then, by Proposition \ref{prop:intClForm}, for some $a \ge 1$, $m^a \in J^a$.  Let $m^a = pn_{g_1}\cdots n_{g_a}$, where $p$ is some monomial and each $g_\ell \in \{1, \ldots, s\}$, perhaps not all distinct.  Let $B_j = \sum {b_{g_aj}}$.  Multiplying by $x_j^{B_j}$, we get 
\[
    x_j^{B_j}m^a = p(x_j^{b_{g_1}j}n_{g_1}\cdots x_j^{b_{g_a}j}n_{g_a}) \in I^a.
\]
Further multiplying the left-hand side by $x_j^{(a-1)B_j}$ yields $x_j^{aB_j}m^a = (x_j^{B_j}m)^a \in I^a$.  Thus, $x_j^{B_j}m \in \overline{I}$.  This holds for all $j \in \{1, \ldots r\}$, so every element of $\overline{J}/\overline{I}$ is annihilated by a power of $\mathfrak m$, so the only associated prime is $\mathfrak m$.
\end{proof}

\begin{lem}\label{lem:satFinite} The modules $(I^n)^{\text{sat}}/I^n$ and $\overline{(I^n)^{\text{sat}}}/\overline{I^n}$ have finitely many monomials.
\end{lem}
\begin{proof} If $I^n = (I^n)^{\text{sat}}$, then the quotient is the zero module, hence contains only finitely many monomials.

For all other cases, by the previous lemma, it suffices to show that the only associated prime of $(I^n)^{\text{sat}}/I^n$ is $\mathfrak m$.


We would like to thank the anonymous referee for the following improved proof of the result: suppose $\mathfrak p$ is an associated prime of $(I^n)^{\text{sat}}/I^n$.  So there exists $m \in (I^n)^{\text{sat}}$ such that $\operatorname{Ann}(m + I^n) = \mathfrak p$.  On the other hand, since $m \in (I^n)^{\text{sat}}$, for each $i \in {1, \ldots, r}$, there is $b_i$ such that $x_i^{b_i}m \in I^n$.  So $x^{b_i} \in \operatorname{Ann}(m + I^n)$, and since $\mathfrak p$ is prime, $x_i \in \mathfrak p$ for all $i$.  Thus, $\mathfrak m \subseteq \mathfrak p$, i.e. $\mathfrak m = \mathfrak p$.

\end{proof}

The above lemma is necessary so that we may use the theory of Ehrhart (quasi)-polynomials.  Now, we may prove Theorem \ref{thm:isdef}:

\begin{proof}[Proof of Theorem \ref{thm:isdef}] First, suppose $\mathfrak m$ is an associated prime.  By Lemma \ref{primdecomppowers}, $I^n = I^{(n)}$ for all $n$, and thus $\isdef_I(n) = 0$.  So from hereon we assume $\mathfrak m$ is not associated.

Let $\{P_i\}$ be the facets of $\SP(I)\setminus \NP(I)$.  Each $P_i$ is bounded, as follows from Remark \ref{rem:satSym} and Lemma \ref{lem:satFinite}.  Furthermore, each $P_i$ is supported on a hyperplane, so they verify an equation, $P_i : E_i(\mathbf u) = c_i$, where $E_i(\mathbf u) = a_{i1}u_{1} + \cdots + a_{ir} u_r$, where we can choose $a_{i\ell}, c_i \in \Z_{\ge 0}$.  As follows from Definition \ref{def:SP}, each face of $n\SP(I)$ is parallel to at least one coordinate hyperplane.  For a given $P_i$, let $\{j_1^{(i)}, \ldots, j_d^{(i)}\}$ denote the indices such that $P_i$ is parallel to the $u_1\cdots u_{j_\ell^{(i)} - 1}u_{j_\ell^{(i)} + 1}\cdots u_r$ coordinate plane (equivalently, such that $a_{ij^{(i)}_\ell} = 0$).  Let $P_{ij}$ denote the projection of $P_i$ into the $u_1\cdots u_{j - 1}u_{j + 1}\cdots u_r$ coordinate plane.  Finally, let $P_{ijs}^n$ be the region in the $u_1\cdots u_{j - 1}u_{j + 1}\cdots u_r$ hyperplane bounded below by $nP_{ij}$ and above by $nP_{ij} + a_{is} \mathbf e_s$.  

We will show that the number of minimal elements of $n\SP(I) \setminus n\NP(I)$ is at most, the number of lattice points in 
\[P\coloneqq \underset{{\substack{j \in \{j_1^{(i)}, \ldots, j_d^{(i)}\} \\ s \in \{1, \ldots, r\}}}}{\bigsqcup} P_{ij^{(i)}s}^n.\]  Indeed, let $m$ be a minimal element in $n\SP(I) \setminus n \NP(I)$.  For each coordinate $v_t$ of $m$, there is a face $P_i$ such that moving back in the $u_t$ direction by one crosses $P_i$.  This means that $c_i n \le E_i(m) < c_in + a_{it}$.  Thus, projecting $m$ into any of the parallel coordinate planes lands it in $P_{ij^{(i)}t}^n$.

By Lemma \ref{lem:satFinite}, each $P_{ij^{(i)}s}^n$ has finite volume, and thus is, at most, an $r-1$ dimensional polytope.  We observe how $P_{ij^{(i)}s}^n$ changes as $n$ increases.  In every direction but the $u_s$-th direction, the length of $P_{ij^{(i)}s}^n$ is scaled by $n$.  However, in the $u_s$-th direction, the length is always $a_{is}$.  Thus, the number of integer points grows like scaling $a_{is}$ different $(r-2)$-dimensional polytopes, which, by Ehrhart (quasi)-polynomials \cite{EhrhartPoly}, gives a (quasi)-polynomial of degree $r-2$.
\end{proof}


\section{Calculating Symbolic Defect with the Symbolic Polyhedron}

We will exploit the symmetries of the ideal $I = (xy, xz, yz) = (x, y) \cap (x, z) \cap (y, z)$ to provide an explicit formula for $\sdef_I(n)$.  While this result can also be calculated with graph-theoretic techniques \cite{DG}, this displays an alternative method.  Furthermore, this method seems to be special to this ideal, but could potentially have some expansion into higher dimensions.

It is first important to note that, since $\dim(R/I) = 1$, $(I^n)^{\sat} = I^{(n)}$ for all $n$.  Furthermore, $I$ satisfies $I^n = \overline{I^n}$ and $I^{(n)} = \overline{I^{(n)}}$ for all $n \ge 1$ (see Lemma \ref{lem:intClFam}), so $\sdef_I(n) = \isdef_I(n)$.  Thus, from Theorem \ref{thm:isdef}, we a priori know that $\sdef_I(n) = O(n)$.

\begin{prop}\label{edgegens} Let $M = I^{(n)}/I^n$.  Then $M = (x^ay^bz^c : \frac32 n \le a + b + c < 2n, \text { and one of }
a = b, b = c, \text { or } a = c)$.  In other words, it is generated by the points on the edges of $n\SP(I) \setminus n \NP(I)$.
\end{prop}

\begin{proof} By Example \ref{ex:easyExample},
\[
    n\SP(I) = \begin{cases} u + v \ge n\\
    u + w \ge n\\
    v + w \ge n\\
    u, v, w \ge 0.
    \end{cases}
\]
Note that the intersection of the first three planes is $(n/2, n/2, n/2)$.  Thus, no element of $n\SP(I)$ has all three entries less than $n/2$.  In particular, an element $(a, b, c) \in n\SP(I)$ satisfies $a + b + c \ge \frac32 n$.  Since $n\NP(I) = n \SP(I) \cap \{ u + v + w \ge 2n\}$, $x^ay^bz^c \in M$ satisfies $\frac32 n \le a + b + c < 2n$.

Now, fix some $n \ge 1$, and let $x^ay^bz^c \in M$.  We will show that there exists a ``path" of unit vectors from an edge lattice point of $n\SP(I) \setminus n\NP(I)$ to $(a, b, c)$.

We use induction on $N = a + b + c$.  If $n$ is even, the base case is that $N = \frac32n$, and $(a, b, c) = (n/2, n/2, n/2)$, which is an lattice point on the edge itself.  If $n$ is odd, $N = \frac32n + \frac 12$, and so $(a, b, c)$ is some permutation of $(\frac 12 (n + 1), \frac 12 (n + 1), \frac 12 (n - 1))$, all of which lie on the edges.

Now, assume the claim holds for $N$ and suppose $(a, b, c) \in M$ satisfy $a + b + c = N + 1$.  If $(a, b, c)$ is on an edge, we are done.  Otherwise, one of $(a - 1, b, c)$, $(a, b-1, c)$, or $(a, b, c-1)$ is in $n\SP(I)\setminus n\NP(I)$, and thus has a path of coordinate unit vectors up to it.  Simply add the appropriate coordinate vector in the appropriate direction.
\end{proof}

\begin{rem} The generation of $M$ in the above proposition is a minimal generation.  Indeed, any monomial corresponding to a point on an edge, of $n\SP(I) \setminus n\NP(I)$, $m$, satisfies, by nature, $x\inv m, y\inv m, z \inv m \not \in I^{(n)}$.  The purpose of Proposition $\ref{edgegens}$ is to demonstrate that they, in fact, generate the entirety of $M$.
\end{rem}

\begin{proof}[Proof of Example \ref{prop:easyExsdef}]  By Proposition $\ref{edgegens}$, this is equivalent to counting lattice edge points on $n\SP(I)\setminus n\NP(I)$.  It is clear that for every increase of $N$ with $\frac32 n \le N < 2n$, there are three more edge points.  If $n$ is even, we get only one point at $N = \frac 32n$.  So, we get $3\left(\frac12 (n - 2)\right) + 1 = \frac 32 n - 2$ for $n$ even, and $3\left(\frac12 (n - 1)\right) = \frac 32 n - \frac 32$ for $n$ odd.
\end{proof}

\section{Calculating Symbolic Defect from the Symbolic Polyhedron, pt. II}

In this section, we use the symbolic polyhedron to determine symbolic defect for the family of ideals, $I = (x^a, y) \cap (y^b, z) \cap (z^c, x) = (xyz, x^az, y^bx, z^cy) \subseteq k[x, y, z]$, utilizing a different method than that used in the previous section.  We choose this family due to the following:

\begin{lem}\label{lem:intClFam} Let $I = (x^a, y) \cap (y^b, z) \cap (z^c, x) = (xyz, x^az, y^bx, z^cy)$.  Then, for all $n > 0$, $I^n = \overline{I^n}$, and $I^{(n)} = \overline {I^{(n)}}$.  Thus, $\sdef_I(n) = \isdef_I(n)$.
\end{lem}
\begin{proof} 
By \cite[Theorem 1.4.10]{HS}, if $I$ is a monomial ideal in $k[x,y,z]$ such that $I$ and $I^2$ are integrally closed, then $I^n = \overline{I^n}$ for all $n > 0$. It is then sufficient to check the $n=1$ and $n=2$ cases. The ideal $I$ is an intersection of integrally closed ideals, hence, from \cite[Corollary 1.3.1]{HS}, it is itself integrally closed. For $n=2$, we study the lattice points in ${\rm NP}(I)$. A sketch of this convex body is given in Figure \ref{fig:npi}.

\begin{figure}
    \centering
        \includegraphics[scale=0.5]{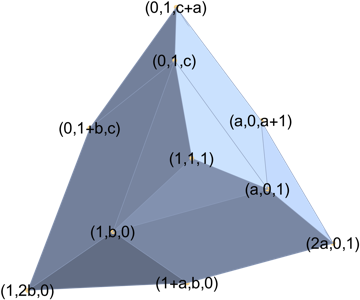}
    \caption{${\rm NP}(I)$}
    \label{fig:npi}
\end{figure}

The vertices of the 3 central triangle faces correspond to the generators of $I$. The other faces stretch to infinity. The angular infinite faces are the intersections with the coordinate planes. Each of the remaining 3 faces represented by parallelograms are parallel to a coordinate line. Concretely, ${\rm NP}(I)$ is described by the \emph{non-negative} solutions of 

$\begin{cases}
   u+av\geq a\\
   v+bw\geq b\\
   cu+w\geq c
\end{cases}$ and $\begin{cases}
   (c-1)u+(a-1)(c-1)v+w\geq (a-1)(c-1)+c\\
   (b-1)(c-1)u+v+(b-1)w\geq(b-1)(c-1)+b\\
   u+(a-1)v+(a-1)(b-1)w\geq (a-1)(b-1)+a.
\end{cases}$

The monomials in $\overline{I^2}$ correspond to the \emph{nonnegative} integer solutions of

$\begin{cases}
   u+av\geq 2a\\
   v+bw\geq 2b\\
   cu+w\geq 2c
\end{cases}$ and 
$\begin{cases}
   (c-1)u+(a-1)(c-1)v+w\geq 2(a-1)(c-1)+2c\\
   (b-1)(c-1)u+v+(b-1)w\geq2(b-1)(c-1)+2b\\
   u+(a-1)v+(a-1)(b-1)w\geq 2(a-1)(b-1)+2a.
\end{cases}$

From the picture, we see that if one of the coordinates is 0, say $w=0$, then we are inside the right triangle with vertex at $(1,b,0)$, meaning a monomial that is a multiple of $(xy^b)^2$, so in $I^2$. We may thus assume that $u,v,w\geq 1$. If they are all at least 2, then we get a monomial that is a multiple of $(xyz)^2$, so again in $I^2$. It remains to treat the case of positive coordinates, but at least one of them is exactly 1. 
By symmetry, without loss of generality, assume $w=1$. We will prove that every possible solution dominates either $(2,b+1,1)$ or $(a+1,b,1)$, meaning that the corresponding monomial in $\overline{I^2}$ is divisible by one of the elements $(y^bx)(xyz)$ or $(x^az)(y^bx)$ in $I^2$.

Write $u=1+U$ and $v=1+V$ with $U,V\geq 0$. We want to prove that $(U,V)$ dominates (componentwise) $(1,b)$ or $(a,b-1)$.
The inequalities we know are
$$\begin{cases}
   U+aV\geq a-1\\
   V\geq b-1\\
   cU\geq c-1\\
   (c-1)U+(a-1)(c-1)V\geq (a-1)(c-1)+c\\
   (b-1)(c-1)U+V+(b-1)\geq(b-1)(c-1)+b\\
   U+(a-1)V\geq (a-1)(b-1)+a.
\end{cases}$$

The fourth is impossible if $c=1$. The third then gives $U\geq 1$. The second has $V\geq b-1$. If $V\geq b$, then $(U,V)$ dominates $(1,b)$ and we are done. The remaining case is $V=b-1$. From the sixth inequality, we obtain $U\geq a$, i.e., $(U,V)$ dominates $(a,b-1)$.

We now prove that $I^{(n)}$ is integrally closed. 
The primary decomposition of $I$ has no embedded primes. For every primary ideal $Q$ in this decomposition we observe that $Q^n$ is integrally closed for all $n$. Intersections of integrally closed ideals are integrally closed, so conclude by Lemma \ref{primdecomppowers}.
\end{proof}

\begin{rem}
    From Definition \ref{def:SP}, it follows that ${\rm SP}(I)$ is described by the nonnegative solutions to the inequalities $\begin{cases}
   u+av\geq a\\
   v+bw\geq b\\
   cu+w\geq c
\end{cases}$.
Their simultaneous equality case is the vertex
\[
P = ( \alpha, \beta, \gamma) = \left( \frac{a(bc - b + 1)}{abc + 1}, \frac{b(ac - c + 1)}{abc + 1}, \frac{c(ab - a + 1)}{abc + 1} \right).
\]
In Figure \ref{fig:spi} we see that the inner triangular faces of Figure \ref{fig:npi} disappear, and the (infinite) parallelogram shaped faces are now shaped like pentagons (in fact they are quadrilaterals with a vertex at infinity in the direction of a corresponding coordinate line) with common vertex $P$.

\begin{figure}
\begin{minipage}{.5\textwidth}
  \centering
  \includegraphics[width=.7\linewidth]{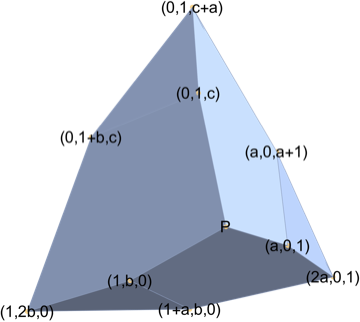}
  \captionof{figure}{${\rm SP}(I)$}
  \label{fig:spi}
\end{minipage}%
\begin{minipage}{.5\textwidth}
  \centering
  \includegraphics[width=.6\linewidth]{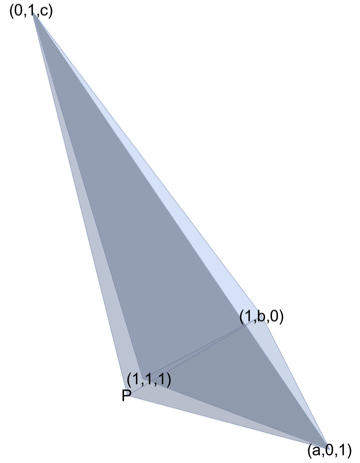}
  \captionof{figure}{${\rm SP}(I)\setminus{\rm NP}(I)$}
  \label{fig:isdef}
\end{minipage}
\end{figure}

Then ${\rm isdef}_I$ is determined by Figure \ref{fig:isdef} where from the larger tetrahedron with vertex $P$ we exclude a smaller tetrahedron with vertex $(1,1,1)$ and same opposite face.



\end{rem}

\smallskip

In conjunction with Example \ref{ex:purePowers}, we will focus on the points ``just inside" the polyhedron, $n \SP(I) \setminus n \NP(I)$.  Upon inspection, this polyhedron forms a tetrahedron with an indented base. The first three planes in the polyhedron meet at the point $P$ above.

We require a few lemmas before our calculation. 

\begin{prop}\label{floorcond} Let $n \in \Z_{\ge 0}$.  A monomial $x_i^{b_i}x_j^{b_j} \in (x_i^{a_i}, x_j^{a_j})^n$ if and only if $\left\lfloor \frac {b_i} {a_i} \right\rfloor + \left\lfloor \frac {b_j} {a_j} \right\rfloor \ge n$.
\end{prop}

\begin{proof}
    One can see that $(x_i^{a_i}, x_j^{a_j})^n = (x_i^{na_i}, x_i^{a_i(n-1)}x_j^{a_j}, \ldots, x_i^{a_i}x_j^{a_j(n-1)}, x_j^{na_j})$.  Thus, $x_i^{b_i}x_j^{b_j} \in (x_i^{a_i}, x_j^{a_j})^n$ if and only if there exists $t \in \{0, 1, \ldots, n\}$ such that $b_i \ge a_i(n-t)$ and $b_j \ge {a_j}t$.  We will show the existence of such a $t$ is equivalent to satisfying the inequality.

    First, assume such a $t$ exists.  Then $n - \frac {b_i}{a_i} \le t \le \frac {b_j}{a_j}$.  In particular, $t$ is an integer between $n - \frac {b_i} {a_i}$ and $\frac {b_j} {a_j}$.  Thus, $\lfloor \frac {b_j} {a_j} \rfloor - \lceil n - \frac {b_i} {a_i} \rceil \ge 0$.  Since $-\lceil -u \rceil = \lfloor u \rfloor$ for any $u \in \mathbb R$, this rearranges to $\left\lfloor \frac {b_i} {a_i} \right\rfloor + \left\lfloor \frac {b_j} {a_j} \right\rfloor \ge n$.

    Assume now that $b_i$ and $b_j$ satisfy the inequality.  Equivalently, we have $\lfloor \frac {b_j} {a_j} \rfloor - \lceil n - \frac {b_i} {a_i} \rceil \ge 0$.  This means there is an integer $\ell$, in between $\frac {b_j}{a_j}$ and $n - \frac {b_i} {a_i}$.  If $\ell$ can only be chosen to be negative, then $b_i \ge a_i(n-\ell) \ge a_in$, but since $b_j \ge 0$ by assumption, $\ell' = 0$ also satisfies $b_i \ge a_i(n - \ell')$, $b_j \ge a_j\ell'$.  Thus, $\ell$ can be chosen to be non-negative.  Similarly, if $\ell$ can only be chosen to be greater than $n$, we have $a_j n \le a_j \ell  \le b_j$.  But then if $\ell'' = n$, we get $b_j > a_j \ell''$ and $b_i \ge a_j(n - \ell'')$, so $\ell$ maybe chosen to be, at most, $n$. So we have the existence of a $t$, i.e. $t = \ell$.
\end{proof}
\begin{cor}\label{symfloorcond}  Let $I = \bigcap_{1 \le i < j \le r} (x_i^{a_{ij}}, x_j^{a_{ji}})$, where $a_{ij} > 0$ for all $i, j$.  Then ${\mathbf x}^{\mathbf b} \in I^{(n)}$ if and only if, for all $1 \le i < j \le r$,  $\left\lfloor \frac {b_i} {a_{ij}} \right\rfloor + \left\lfloor \frac {b_j} {a_{ji}} \right\rfloor \ge n$.

\end{cor}
\begin{proof}
By Lemma \ref{primdecomppowers}, $I^{(n)} = \bigcap_{1 \le i < j \le r} (x_i^{a_{ij}}, x_j^{a_{ji}})^n$.  Apply Proposition \ref{floorcond} to each of the ideals in the intersection.
\end{proof}

Let $m = {\mathbf x}^{\mathbf b}$ be a minimal generator of $I^{(n)}$.  By definition, $m \in I^{(n)}$, but for all $i \in \{1, \ldots, r\}$, $x_i\inv m \not\in I^{(n)}$ or is undefined, depending on if $x_i$ divides $m$.  This entails that, for each $i$, there is at least one $j$ such that $x_i\inv m \not \in (x_i^{a_{ij}}, x_j^{a_{ji}})^n$.  We then use Corollary \ref{symfloorcond} to get the following:

\begin{prop}\label{modcon} Let $I$ be as in Corollary \ref{symfloorcond}.  If $m = {\mathbf x}^{\mathbf b} \in I^{(n)}$ but $x_i\inv m \not \in I^{(n)}$, then, for at least one $j$, $b_i \equiv 0 \pmod {a_{ij}}$.
\end{prop}

\begin{proof} By Corollary \ref{symfloorcond}, we have that $\left\lfloor \frac {b_i} {a_{ij}} \right\rfloor + \left\lfloor \frac {b_j} {a_{ji}} \right\rfloor \ge n$ for all $j \ne i$, but, there is a $j'$ such that $\left\lfloor \frac {b_i - 1} {a_{ij'}} \right\rfloor + \left\lfloor \frac {b_{j'}} {a_{j'i}} \right\rfloor < n$.  It follows that $\left\lfloor \frac {b_i-1} {a_{ij'}} \right\rfloor < \left\lfloor \frac {b_i} {a_{ij'}} \right\rfloor$, which can only occur if $\frac {b_i}{a_{ij'}}$ is an integer.
\end{proof}

We now have the following:

\begin{prop} Let $I$ be as in Corollary \ref{symfloorcond}.  If $m = {\mathbf x}^{\mathbf b}$ is a minimal generator of $I^{(n)}$, then, for each $1 \le i \le r$, there is some $j \ne i$ such that in the first orthant of $\mathbb R^r$, $\mathbf{b}$ lies on one of finitely many hyperplanes parallel to $\frac {u_i}{a_{ij}} + \frac {u_j}{a_{ij}} = n$.  Moreover, the number of possible planes on which $\mathbf b$ may lie is constant.
    
\end{prop}

\begin{proof} Fix $i \in \{1, \ldots, r\}$.  Since $m$ is minimal, choose a $j$ such that $\left\lfloor \frac {b_i - 1} {a_{ij}} \right\rfloor + \left\lfloor \frac {b_{j}} {a_{ji}} \right\rfloor < n$.  By Corollary \ref{symfloorcond} and Proposition \ref{modcon}, it follows that $n - \ell \le \left\lfloor \frac {b_{j}} {a_{ji}} \right\rfloor < n - \ell + 1$, where $b_i = a_{ij}\ell$.  Thus, the value of $b_j$ has only finitely many possibilities, namely $a_{ji}$ possibilities, and thus only finitely many planes parallel to $\frac {u_i}{a_{ij}} + \frac {u_j}{a_{ij}} = n$ can contain a minimal generator.    
\end{proof}

\begin{rem}\label{rem:doubleMod0} Note that, using the above notation, if $b_i \equiv 0 \pmod {a_{ij}}$ and $b_j \equiv 0 \pmod {a_{ji}}$, then we get $\frac 1{a_{ij}}(a_{ij}\ell_1) + \frac 1{a_{ji}} (a_{ji}\ell_2) = \ell_1 + \ell_2 \ge n$.  By minimal generation, we get that $\ell_1 + \ell_2 - \frac 1{a_{ij}} < n$ or $\ell_1 + \ell_2 - \frac 1{a_{ji}} < n$.  But since $\frac 1{a_{ij}}, \frac 1{a_{ji}} \le 1$, it follows that $\ell_1 + \ell_2 = n$.  Thus, this minimal generator lies on the face $\frac {u_i}{a_{ij}} + \frac {u_j}{a_{ij}} = n$.
    
\end{rem}

We are now ready to calculate $\sdef_I(n)$ for our family.  We have the following formulation:

\begin{prop}\label{prop:facegens} Let $I = (x^a, y) \cap (y^b, z) \cap (z^c, x) = (x^az, y^bx, z^cy, xyz)$.  If $m = x^uy^vz^w$ is a minimal generator of $I^{(n)}/I^n$, then:
\begin{enumerate}
    \item If $\frac 1a u + v = n$, $m = x^{an-av}y^vz^{\max\{n - \left\lfloor \frac v b \right\rfloor, cn + cav - can\}}$ for some $1 \le v \le n -1$,
    \smallskip
    
    \item If $\frac 1b v + w = n$, $m = x^{\max\{n - \left\lfloor \frac w c \right\rfloor, an - abn + abw\}}y^{bn - bw} z^w$ for some $1 \le w \le n - 1$,
    \smallskip
    
    \item If $\frac 1c w + u = n$, $m = x^uy^{\max\{n - \left\lfloor \frac u a \right\rfloor, bn - bcn + bcu\}}z^{cn - cu}$ for some $1 \le u \le n - 1$.
\end{enumerate}
    
\end{prop} 
\begin{proof} We will just prove (1), as (2) and (3) are analogous.  So assume that $u = an - av$.

The defining inequalities for $n\NP(I)$ is as follows:
\[
    n\NP(I) = \begin{cases} u + (a-1)v + (a-1)(b-1)w \ge (ab - b + 1)n\\
    (c-1)u + (a-1)(c-1)v + w \ge (ac - a + 1)n\\
    (b-1)(c-1)u + v + (b-1)w \ge (bc - c + 1)n\\
    \frac1a u + v \ge n\\
    \frac 1b v + w \ge n\\
    \frac 1c w + u \ge n\\
    u, v, w \ge 0
    \end{cases}.
\]
The defining inequalities for $n\SP(I)$ are the latter four equations defining $n\NP(I)$.

From the fourth and seventh inequalities, we see $0 \le v \le n$.  If $v = 0$, then $u = an$ and $w \ge n$.  But this means that all of the first three inequalities are satisfied, so $m \in n\NP(I)$.  So, if $v = 0$, $m$ cannot be a minimal generator in $n\SP(I) \setminus n\NP(I)$.  For similar reasons, no minimal generator can have $v = n$.

To see that $m \in n \SP(I)$, observe:

\begin{align*}
    \frac1a(an - av) + v & = n\\
    \frac1b v + \max\{ n - \left\lfloor \frac v b \right\rfloor, cn + cav - can \} & \ge \frac 1b v + n -  \left\lfloor \frac v b \right\rfloor \ge n,\\
    \frac1c \max\{ n - \left\lfloor \frac v b \right\rfloor, cn + cav - can \} + an - av & \ge \frac 1c (cn + cav - can) + an - av  = n.\\
\end{align*}
Now, to see $m \not \in n \NP(I)$, we consider it by cases.  
If the $w$-coordinate of $m$ is $n - \left \lfloor \frac v b \right \rfloor$, $w \le n$.  Thus, with the first inequality, we see
\[
    an - av + (a - 1)v + (a-1)(b-1)w \le an - av + (a-1)v + (a-1)(b-1)n = (ab - b + 1)n - v < (ab - b + 1)n.
\]

If $w = cn + cav - can$, since $v < n$, considering the third inequality yields:
\[
    a(b-1)(v-n) + bcn - cn + v \le bcn - cn + v < (bc - c + 1)n.
\]

Now, note $\frac 1a (an - av - 1) + v = n - v - \frac 1a + v = n - \frac 1a < n$, so $x\inv m \not\in I^{(n)}/I^n$.  Also, $\frac 1a(an - av) + v - 1 = n - v + v - 1 = n - 1 < n$, so $y\inv m \not \in I^{(n)}/I^n$.  Finally, to see $z \inv m \not \in I^{(n)} / I^n$, we take it in cases.  If $n - \left\lfloor \frac v b \right\rfloor \le cn + cav - can$,
\[
    \frac 1c (cn + cav - can - 1) + an - av = n - \frac 1c < n.
\]
If $n - \left\lfloor \frac v b \right\rfloor \ge cn + cav - can$,
\[
    \frac vb + n - \left\lfloor \frac v b \right\rfloor - 1 < n + 1 - 1 = n.
\]

Since all monomials of this form are minimal generators, all other monomials with first two exponents $an - av$ and $v$ will have a final exponent, at least, $\max\{n - \left\lfloor \frac v b \right\rfloor, cn + cav - can\}$, so there are no other possible minimal generators residing on the face, $\frac1au + v = n$.
\end{proof}

We have now established a (potentially strict) subset of the generators of $I^{(n)}/I^n$.  Let $M(n)$ be this set of generators.  However, even if this is a strict subset, we have the following:

\begin{prop}\label{prop:nonFaceGen} Let $I$ be as in Proposition \ref{prop:facegens}. Suppose $M(n)$ is a strict subset of the minimal generators of $I^{(n)}/I^n$.  That is, suppose there is a minimal generator not residing on a face.  Then this minimal generator must have the exponent vector, $R_n = (\lceil n \alpha \rceil, \lceil n \beta \rceil, \lceil n \gamma \rceil)$. That is, the minimal generator must be $x^{\lceil n \alpha \rceil}y^{\lceil n \beta \rceil} z^{\lceil n \gamma \rceil}$.
\end{prop}

\begin{proof}

Suppose $(u, v, w)$ are the coordinates of a minimal generator of $I^{(n)}/I^n$ that does not reside on any face.  By Remark \ref{rem:doubleMod0}, $a$ does not divide $u$, $b$ does not divide $v$, and $c$ does not divide $w$.  It follows that $n < \frac 1au + v < n + 1$, $n < \frac 1b v + w < n + 1$, $n < \frac 1c w + u < n + 1$.  Thus, we have $v = n - \left\lfloor \frac ua \right\rfloor$, $w = n - \left\lfloor \frac v b \right\rfloor$, $u = n - \left\lfloor \frac w c \right\rfloor$.

Let's focus on the possible values of $v$.  It is evident that $1 \le v \le n$.  If $1 \le v \le n - 1$, then $x^{an - av}y^vz^{\max\{n - \left\lfloor \frac v b \right\rfloor, cn + cav - can\}}$ has already been shown to be a generator.  If $n - \left\lfloor \frac v b \right\rfloor \ge cn + cav - can$, then we have two generators with two identical coordinates, so they must be the same generator.  Otherwise, $n - \lfloor \frac v b \rfloor < cn + cav - can$.  With some rearranging, we get $\left\lfloor \frac{(abc + 1)v}b\right\rfloor > (ac - c + 1)n$.  It follows that $\lceil n \beta \rceil \le v \le n$.  By analogous arguments, we have $\lceil n \alpha \rceil \le u \le n$, and $\lceil n \gamma \rceil \le w \le n$.  It follows that $(u, v, w)$ is generated by $R_n \in n\SP(I)$, and thus, by minimality, must be equal to $R_n$.
\end{proof}

\begin{rem} The point $R_n$ above may not be a minimal point of $n\SP(I)$.  For example, if $a = 2$, $b = 3$, and $c = 4$, $R_4 = (4, 3, 4)$, which is generated by $(3, 3, 4) \in n\SP(I)$.   Furthermore, $R_n$ itself may not be in $n\SP(I) \setminus n\NP(I)$.  It also may reside on a face.  However, the assumptions of Proposition \ref{prop:nonFaceGen} specify that a minimal point not on the faces exists, and when this is the case, this generator is $R_n$.
\end{rem}

\begin{prop} \label{prop:faceGensGrowth} Let $I$ be as in Proposition \ref{prop:facegens}, and $M(n)$ be the set of generators described in Proposition \ref{prop:facegens}.  Then $\#M(n) \sim (\alpha + \beta + \gamma)n$.
    
\end{prop}
\begin{proof} We use inclusion-exclusion principle.  Clearly, the individual types of the monomials given in Proposition \ref{prop:facegens} add up to $3(n-1)$, so we need to consider the monomials that are multiple types.  Note that from hereon, we are only interested in the leading coefficient, so we drop the -1 from $3(n-1)$.

We focus on the number of monomials that are type (1) and type (3).  This is the same as the pairs $(u, v)$ such that $1 \le u, v \le n - 1$, $an - av = u$, $v = \max\{ n - \left\lfloor \frac u a \right\rfloor, bn - bcn + bcu \}$, and $\max\{ n - \left\lfloor \frac v b \right\rfloor, cn + cav - can\} = cn - cu$.  We can determine this by determining the $v$ such that $1 \le v \le n-1$, $1 \le an - av \le n - 1$, $ v \ge bn - bcn + bc(an - av)$ and $n - \left\lfloor \frac v b \right\rfloor \le cn + cav - can$.  With some simplification, the latter two are satisfied when $v \ge n \beta$. 

Furthermore, every $n\beta \le v \le n -1$ satisfies $an - av \ge 1$.  Finally, if $v \ge n \beta$, then $a(n - v) \le a(n - n \beta) = an(1-\beta) = an \left(\frac \alpha a\right) = \alpha n < n$.  Since $an - av$ is an integer strictly less than $n$, we have $an - av \le n - 1$.

Thus, there are $n - 1 - \lceil n \beta \rceil$ possible generators of types (1) and (3).  This grows like $(1-\beta)n$.  Applying analogous arguments for the rest, we conclude that leading term of $\#M(n)$ is $3n - n(1 - \alpha + 1 - \beta + 1 - \gamma) = n(\alpha + \beta + \gamma)$.
\end{proof}

We obtain the immediate conclusion:

\begin{proof}[Proof of \ref{ex:sdefCalc}.(i)] Apply Propositions \ref{prop:nonFaceGen} and \ref{prop:faceGensGrowth}.
\end{proof}

From the data, we can verify Example \ref{ex:sdefCalc}(ii).

\begin{proof}[Proof of \ref{ex:sdefCalc}.(ii)]  Observe $\sdef_I(n) = \#M(n) + 0$ or $\#M(n) + 1$ depending on whether the point $R_n = (\lceil n \alpha \rceil, \lceil n \beta \rceil, \lceil n \gamma \rceil)$ is a minimal generator off of a face or not.  We will show that this occurs for $n$ if and only if it occurs for $n + \ell(abc + 1)$ for any $\ell \ge 1$. 
 Indeed, consider the first coordinate of $R_{n + \ell(abc + 1)}$:
\[
    \lceil (n + \ell(abc + 1))\alpha \rceil = \lceil n\alpha + \ell(a(bc - b + 1)) \rceil = \lceil n \alpha \rceil + a\ell(bc - b + 1).
\]
Thus, the first coordinate of $R_n$ and $R_{n + \ell(abc + 1)}$ are congruent modulo $a$.  If we assume $R_n$ is a minimal generator not on a face, then $\lceil n \alpha \rceil \not\equiv 0 \pmod a$.  This occurs if and only if the first coordinate $R_{n + \ell(abc + 1)}$ is not a multiple of $a$.  Applying the same argument for the other coordinates, $R_n$ is a minimal generator not on a face implies $R_{n + \ell(abc + 1)}$ is not on a face.  Finally, to see $R_{n + \ell(abc + 1)}$ is a minimal generator, consider the inequality $n \le \frac 1a \lceil n\alpha \rceil + \lceil n \beta \rceil < n + 1$, which holds  because $R_n$ is a minimal generator not on a face.  Then $n + \ell(abc + 1) \le \frac1a(\lceil (n + \ell(abc + 1))\alpha \rceil) + \lceil (n + \ell(abc + 1))\beta \rceil < n + \ell(abc + 1) + 1$ via simple addition, and the other inequalities hold too.

\end{proof}

\begin{ex} Let $I = (x^2, y) \cap (y^3, z) \cap (z^4, x) = (xyz, x^2z, y^3x, z^4y)$.  Then, $P = \left(\frac45, \frac 35, \frac 45\right)$, and $R_n$ is a minimal generator not on a face if and only if $n \equiv 3 \pmod 5$.  In particular,
\[
    \sdef_I(n) = \begin{cases} 0 & n = 1\\
    \frac {11}5 n - 2 & n \equiv 0 \pmod 5\\
    \frac{11}5 n - \frac 65 & n \equiv 1 \pmod 5\text{ and } n \ne 1 \\
    \frac {11}5 n - \frac 75 & n \equiv 2 \pmod 5\\
    \frac {11}5 n - \frac 35 & n \equiv 3 \pmod 5\\
    \frac {11}5 n - \frac 45 & n \equiv 4 \pmod 5\\
    \end{cases}.
\]

Note that, in this case, although $abc + 1 = 25$ is a quasi-period, it is not the minimal quasi-period of $\sdef_I$.
    
\end{ex}

\begin{rem} In Drabkin and Guerrieri's proof that (under the appropriate conditions) $\sdef_I$ is eventually quasi-polynomial, he determines that a quasi-period is the least common multiple of the degrees of the generators of $\mathcal R_s(I)$ as an $R$-algebra \cite{DG}.  However, the quasi-period calculated from Example \ref{ex:sdefCalc}(ii), $abc + 1$, looks nothing like a least common multiple.  This suggests perhaps a deeper, number-theoretical relationship between $a$, $b$, $c$, and  the degrees of the generators of $\mathcal R_s(I)$.
    
\end{rem}

\section{Open Questions}
We have seen some results on the growth of $\sdef_I(n)$, but we still have many unanswered questions about the function.  These include, but are not limited to:
\begin{question} Given an arbitrary monomial ideal $I$,
    \begin{enumerate}[1.]
        \item What is the degree of $\sdef_I(n)$?
        \smallskip
        
        \item What are the coefficients of $\sdef_I(n)$?
        \smallskip
        
        \item Are the leading coefficients of the quasi-polynomial on each branch equal?
        \smallskip
        
        \item What is the minimum period for the symbolic defect function?
        \smallskip
        
        \item How large does $n$ have to be before adhering to the quasi-polynomial structure?
\end{enumerate}
\end{question}

One can ask the very same questions about $\isdef_I(n)$.  Furthermore, as stated in Section 4, there is no immediately known relationship between growth rates of $\sdef_I(n)$ and $\isdef_I(n)$.  So, we also ask,

\begin{question} How do the growth rates of $\sdef_I(n)$ and $\isdef_I(n)$ compare?  For what ideals does $\isdef_I(n) = O(\sdef_I(n))$?
\end{question}

Lastly, the methods of Section 5 and 6 seem to be specific, but the natural question is,

\begin{question} Can the methods of Section 5 and/or 6 be expanded upon to calculate $\sdef_I(n)$ or $\isdef_I(n)$ for more general ideals?
\end{question}



\section{Acknowledgements}
I would like to especially thank my advisor, Mihai Fulger, for introducing the project to me and helping me along the way.  My research was partly funded by his Simons Collaboration Grant no.~579353.
I would also like to thank Benjamin Drabkin, Michael DiPasquale, Federico Galetto, Alexandra Seceleanu, and Adam Van Tuyl for their helpful feedback and suggestions.
We utilized Macaulay2, Mathematica, and Grapher to help develop this project.

\end{document}